\documentclass[12,fleqn]{article}
\usepackage{a4wide,caption,epsfig,url}

\usepackage{amsthm}
\usepackage{amsfonts}
\usepackage{amsmath}
\usepackage{amssymb}
\usepackage{verbatim}
\usepackage{array}

\newtheorem{theorem}{Theorem}[section]
\newtheorem{lemma}[theorem]{Lemma}

\newtheorem{proposition}[theorem]{Proposition}
\newtheorem{example}[theorem]{Example}

\begin{document}

\title{High Rate LDPC Codes from Difference Covering Arrays}

\author{D.~Donovan\thanks{D. Donovan is with the Centre for Discrete Mathematics and Computing,  The University of Queensland, St Lucia 4072, Australia. Email: dmd@maths.uq.edu.au}%
,  A.~Rao\thanks{A. Rao is with the School of Science (Mathematical Sciences), RMIT Univerity, Melbourne 3000, Australia. Email: asha@rmit.edu.au}%
  , and E.~\c{S}ule~Yaz\i c\i\thanks{E. \c{S}ule Yaz\i c\i ~~is with the Department of   Mathematics, Ko\c{c} University,  Rumelifeneri Yolu, 34450, Sar\i yer, Istanbul,  TURKEY.  Email: eyazici@ku.edu.tr}.}
%


\maketitle

\begin{abstract}
This paper presents a combinatorial construction of low-density parity-check (LDPC) codes from difference covering arrays. While the original construction by Gallagher was by randomly allocating bits in a sparse parity-check matrix, over the past 20 years researchers have used a variety of more structured approaches to construct these codes, with the more recent constructions of well-structured LDPC coming from balanced incomplete block designs (BIBDs) and from Latin squares over finite fields. However these constructions have suffered from the limited orders for which these designs exist. Here we present a construction of LDPC codes of length $4n^2 - 2n$ for all $n$ using the cyclic group of order $2n$. These codes achieve high information rate (greater than 0.8) for $n \geq 8$, have girth at least 6 and have minimum distance 6 for $n$ odd. 
\end{abstract}


\section{Introduction}

Since the introduction of low density parity check (LDPC) codes by Gallager in 1962 \cite{Gallager62}, and subsequent reintroduction by McKay and Neal \cite{MacKayNeal97,MacKay99} in the 1990s, researchers have expended much effort on the design and construction of these codes. The main attraction of these codes is their ability to achieve rates close to the Shannon limit. The codes originally designed by Gallagher were pseudorandom and good LDPC codes were found mainly through computer searches \cite{Gallager62, Gallager63}. Such randomly generated codes do not lend themselves to encoding very easily due to the lack of structure. In addition, it is hard to determine the  minimum distances of these codes. 

Since that time there have been many constructions for LDPC codes using a variety of different approaches, such as those constructed from finite geometries by Kou, Lin and Fossorier \cite{KouLinFoss01}, those with group-structure \cite{TanSriFuja} and those from Ramajunan graphs \cite{RosenVont01}. More recently, starting with Johnson and Weller \cite{JohnWell01} researchers have looked to combinatorial designs to construct well-structured LDPC codes, with  Vasic and Milenkovic \cite{VasMil04}  using balanced incomplete block designs (BIBDs) and Zhang et al \cite{Zhang10} constructing quasicyclic LDPC codes based on Latin squares over finite fields. The main constraint with these constructions is the very restricted orders for which these designs exist. For example, Zhang et. al's \cite{Zhang10} construction uses sets of mutually orthogonal Latin squares which are known to exist only for orders a power of a prime. In this paper we give an alternate construction to that given by Zhang et. al \cite{Zhang10}. The pseudo orthogonal  Latin squares that we use in our construction are much more prolific than the combinatorial structures used previously. The LDPC codes  of length $4n^2 - 2n$ constructed here, using the cyclic group of order $2n$, exist for all $n \in \mathbf{Z}$. In addition, unlike past combinatorial constructions, we are able to give explicit algebraic expressions for the rate of the code as well as the minimum distance of the code. These codes achieve high information rate ($\geq 0.8$) for $n \geq 8$ (code length 240) and for $n$ odd, have minimum distance 6, whereas in the past the codes have had to be of length approximately 5000bits of achieve such rates and it has been very difficult to establish the minimum distance of the codes. We start with the preliminary definitions.

A $(m,\gamma, \rho)$-{\em regular binary LDPC code} ${\cal C}$ \cite{Gallager62} of block length $m$ is given by the null
space of a sparse parity-check matrix $H$ over GF(2) where the
row weight ($\rho$) and column weight ($\gamma$)  are both constant.  The {\em distance} of the code is taken to be the minimum Hamming distance between any two code words. Since the zero vector is a codeword the distance of the code  is equal to the minimum weight over all the codewords. The aim here is to construct codes with  $\rho$ and $\gamma$ small compared to the code length $m$, and that satisfy the {\em RC-constraint}; that is,  the inner product of any two distinct rows or any two distinct columns of the  parity check
matrix $H$ is less than or equal to 1.

Parity-check matrices are often visualised as {\em Tanner  bipartite graphs} with vertex set  $B\cup V$ where $B$ is comprised of code bits and $V$ is comprised of parity-check equations. An edge $\{U,w\}, U\in B$ and $w\in V$, exists in this bipartite graph if and only if $U$ is a term in the check equation $w$.
The RC-constraint on $H$ implies that the Tanner graph  of the LDPC
code has no cycles of  length less than 6, giving a Tanner graph of  girth  at least 6 \cite{Zhang10}.

It is known that the incidence matrix of a balanced incomplete block design (BIBD) or partially balanced incomplete block design (PBIBD) can be used to construct a parity-check matrix for an  LPDC code, see for example \cite{VasMil04}.
A $(v,c,\lambda)$ {\em balanced incomplete block design (BIBD)} is an ordered pair $(V,B)$, where $V$ is a point set of size $v$ and $B$ is a set of, $b$, $c$-subsets\footnote{It is usual for the subset size to be denoted by $k$, but we use $c$ instead, since in coding theory, $k$ is normally reserved for the dimension of the
code.} of $V$, called {\em blocks}, chosen in such a way that each pair of elements of $V$ occurs together in $\lambda$ blocks. For a BIBD it can be shown that each element of $V$ occurs in $r=\lambda(v-1)/(c-1)$ blocks with the total number of blocks given by $b=\lambda (v^2-v)/(c^2-c).$  A BIBD is said to be {\em resolvable} if the $b$ blocks can be partitioned into $b/v$ {\em parallel classes} such that every point of $V$ occurs in some block of each parallel class.

In a partially balanced incomplete block design (PBIBD) the value of $\lambda$ may vary across the pairs $x,y\in V$. A PBIBD is said to be {\em resolvable} if its $b$ blocks can be partitioned into $b/v$ parallel classes in such a way that every point of $V$ occurs in one block of the parallel class.

The incidence matrix of a BIBD  (PBIBD)  $(V,B)$ with $B=\{B_1,\dots, B_b\}$, is a $v\times b$ matrix $A=[a_{ij}]$ such that
\begin{eqnarray*}
a_{ij}&=\left\{\begin{array}{ll}
1,&\mbox{if }i\in B_j,\\
0,&\mbox{otherwise.}
\end{array}\right.
\end{eqnarray*}
For more details on the construction of LPDC codes from block designs see \cite{Mahadevan} or \cite{VasMil04}. In this setting, the Tanner graph contains an edge from $U$ to $w$ if point  $w\in V$ occurs in block  $U\in B$. In particular, Vasic and Milenkovic \cite{VasMil04}  construct parity-check matrices with column sum  $\gamma=c=3$ using the incidence matrix, $A$  of cyclic $(v,3,1)$ BIBDs. It is known that cyclic $(v,3,1)$ BIBDs exist for all $v\equiv 1\,(\mbox{mod }6)$. However any cyclic $(v,3,1)$ BIBD containing a Pasch configuration (a set of 4 blocks of the form $\{0,y,z\},\{0,u,w\},\{x,y,u\},\{x,z,w\}$ will have minimum distance 4.  Vasic and Milenkovic  provide three constructions for Pasch free $(v,3,1)$ BIBDs but these constructions are for the restricted  orders:    $v$ is a power of a prime of the form $v\equiv 1\, (\mbox{mod }6)$ or $v \equiv 7\,(\mbox{mod }12)$,  or $v = p^n$ where $p \equiv 7\,(\mbox{mod }12)$.  Vasic and Milenkovic \cite{VasMil04} suggest that not enough of these designs are known to allow for flexibility in code length and column weight especially for constructing codes with high-rate and/or moderate length codes. In \cite{VasMil04} the {\em rate} of the code constructed is defined to be   $R=(b-{\rm rank}(A))/b$.

  Zhang et.al. \cite{Zhang10} have drawn the connection between PBIBD and Latin squares and used these to construct $(m,\gamma, \rho)$-regular binary LDPC codes, with information rate at least $(\rho-\gamma)/\rho$. While this construction technique delivers useful LPDC it is restricted by the fact that it uses sets  of mutually orthogonal Latin squares which are constructed from finite fields of order a power of a prime.

 In this paper we provide an alternative construction to that of   Zhang et.al. \cite{Zhang10}, based on the cyclic group of even order $2n$. This construction has the added advantage that the arithmetic is  simplified as it is taken as addition modulo $2n$. This simplified setting allows us to provide transparent arguments for the rank of the parity-check matrix and the size of the null space, as well as the distance and rate of the code. Such explicit formulations for the distance and rate of the code have not be possible before. The main ingredient for our constructions are difference covering arrays and pairs of pseudo orthogonal Latin squares, both of which are defined in the next two sections.

\section{Difference Covering Arrays and Partially Balanced Incomplete Block Designs}

Let $(G,*)$ be an abelian group.
 If $G$  acts on a set $X$, then the set $O_x = \{gx : g \in G\}, x \in X$, is called the {\emph orbit} of $x$. A \emph{difference covering array}  DCA$(k,\eta;m)$ is an
 $\eta \times k$ matrix ${\cal Q}=[q(i,j)]$ with entries from $G$  such that,
\begin{itemize}
\item  for all distinct pairs of columns  $0\leq j,j^\prime\leq k-1$  the difference set
 $\Delta_{j,j^\prime}=\{q(i,j)*(q(i,j^\prime) )^{-1} \mid 0\leq i\leq \eta-1\}$
 contains every element of $G$ at least  once.
\end{itemize}
For the remainder of this paper it is assumed that $(G,*)$ is the cyclic group $({\Bbb Z}_{2n},*)$ where $n\geq 2$ and $*$ is addition modulo $2n$. In this case the difference covering array is  said to be {\em cyclic}.
Since the set $\Delta_{j,j^\prime}$ is not altered by adding a constant vector over ${\Bbb Z}_{2n}$  to any row or any column,  we may assume that the first row and first column contain only $0$. To reduce notation we delete the first row and only work with arrays  $Q=[q(i,j)]$ that satisfy the following properties:
\begin{itemize}
\item[{\bf P1.}]  The first column of $Q$ contains only $0$ and the remaining columns  contain each  entry  of  ${\Bbb Z}_{2n}$ precisely once, and
\item[{\bf P2.}]  For all pairs of distinct columns, $j$ and $j^\prime$, $ j \neq 0 \neq j^\prime$,  $\Delta_{j,j^\prime}=\{q(i,j)-q(i,j^\prime)  \mbox{ mod }2n\mid 0\leq i\leq n\}={\Bbb Z}_{2n}\setminus\{0\}$.
\end{itemize}
As it is easy to reinstate the row of all zeros, we will abuse the definition and  call $Q$ a   DCA$^*(k,2n;2n)$ with rows and column labelled $0,\dots 2n-1$ and $0,\dots, k-1$ respectively.

The following example of a cyclic DCA$^*(4, 6;6)$ that satisfies  P1 and P2, is taken from \cite{RSS}.
$$
Q^T=\left[\begin{array}{ccccccc}
0&0&0&0&0&0\\
0&1&2&3&4&5\\
1&3&5&0&2&4\\
3&0&4&1&5&2\\

\end{array}\right]
$$

Using  properties P1. and P2. it can be shown that for any cyclic DCA$^*(k,2n;2n)$ and any pair of distinct columns  $j$ and $j^{\prime}, j \neq 0 \neq j^\prime$,
$\Delta_{j,j^\prime}=\{1,2,\dots,n,n,\dots, 2n-1\}$
with repetition retained, see \cite{DDHKR} for a proof. For results on  difference covering assays see \cite{Yin1, Yin2},  but for general $k$ not a lot is known. However it is known that when $k=3$ the matrix
$X=[x(j,g)]$, where
\begin{eqnarray}
x(j,g)&=\left\{\begin{array}{ll}
0,&\mbox{if }g=0\\
j,&\mbox{if }g=1 \\
\left\{\begin{array}{ll}
2j+1   & \mbox{for } 0 \leq j \leq n -1,\\
2(j-n)  & \mbox{for } n \leq j \leq  2n -1,
\end{array}\right. & \mbox{if }g=2
\end{array}
\right.\label{DCA}
\end{eqnarray}
forms a DCA$^*(3,2n;2n)$ satisfying properties P1 and P2.  In addition, if $k=4$ then there exists a DCA$^*(4,2n;2n)$  for all $n\geq 2$ except possibly $n=73$  \cite{DDHKR,LvR}. But it should be noted that not all
these difference covering arrays are cyclic.

DCAs can be used to construct PBIBDs, where pairs of points occur in 0 or 1 block. While any DCA$^*(3,2n;2n)$ satisfying P1. and P2. can be used to construct PBIBDs, here we use $X=[x(j,g)]$ as defined in Equation \ref{DCA}, to  construct a cyclic PBIBD with the  $2n-1$  ordered starter blocks:
\begin{eqnarray*}
SB_j &=& (x(j,0), x(j,1), x(j,2))=(0, j, x(j,2)), \mbox{ for } j \in {\Bbb Z}_{2n} \setminus \{n\},
\end{eqnarray*}
and develop each of these into orbits giving sets of blocks:  $0 \leq a \leq 2n - 1$,
\begin{eqnarray*}
B_{ja} \hspace{-2mm}&=&\hspace{-2mm} \{ a, j+a\,(\mbox{mod }2n)+2n, x(j,2)+a\,(\mbox{mod }2n)+4n\}  \mbox{ where } (0, j, x(j,2)) = SB_j.
\end{eqnarray*}
  Then
 \begin{eqnarray*}
O_{j} &=&  \{B_{ja},  0\leq a\leq 2n-1\}
   \end{eqnarray*}
for $j \in {\Bbb Z}_{2n}\setminus \{n\}.$
  Note that the starter block $(0,n,0)$ and the corresponding orbit have been omitted.
  Finally we take the union of the orbits   to obtain the blocks of a block design.
Recall that  Property P2 of the difference covering array implies that for each $g, g' \in \{0, \ldots, k-1\}$ if there exists $j, j'\in {\Bbb Z}_{2n}\setminus\{n\}$ such that $x(j,g) - x(j, g^\prime) = x(j^\prime,g) - x(j^\prime, g^\prime)$, then $j=j^\prime$. Thus the set
\begin{eqnarray*}
{\cal B}=\cup_{j\in {\Bbb Z}_{2n}\setminus\{n\}}O_{j}
  \end{eqnarray*}
  is a set of  $4n^2-2n$ $3$-subsets (blocks) of ${\Bbb Z}_{6n} = V$ with the property that pairs $y,z\in {\Bbb Z}_{6n}$ occur together in $\lambda_{y,z}$ blocks where
\begin{eqnarray*}
\lambda_{y,z}&=&\left\{
\begin{array}{ll}
0, & \mbox{if }y,z\in\{0,\dots, 2n-1\};\ y,z\in\{2n,\dots, 4n-1\};\ y,z\in\{4n,\dots, 6n-1\},\\
0, & \mbox{if }z=y+3n\mbox{ or } z=y+4n\mbox{ and } y\in\{0,\dots, 2n-1\},\\
0, & \mbox{if }z=y+3n\mbox{ and }y\in\{2n,\dots, 3n-1\},\\
0, & \mbox{if }z=y+n\mbox{ and } y\in\{3n,\dots, 4n-1\},\\
1, &\mbox{ otherwise}.
\end{array}\right.
\end{eqnarray*}
Further each orbit defines a parallel class.
Hence $(Z_{6n},{\cal B})$ is a PBIBD$(6n,3)$ with $\lambda \leq 1$ that is resolvable into $2n-1$ parallel classes.

We now follow the work of Vasic and Milenkovic \cite{VasMil04} and construct an incidence matrix $H=[h(i,j)]$ where the columns are indexed by the blocks  $B_{ja}\in{\cal B}$ and set
\begin{eqnarray}
h(i,ja)&=&\left\{
\begin{array}{ll}
1,&\mbox{if }i\in B_{ja},\\
0,&\mbox{otherwise}.
\end{array}\right.\label{IM}
\end{eqnarray}

In Section \ref{LI} we will identify a set of $6n-2$ columns of $H$ that are linearly independent and then use this information to obtain the dimension of the null space and hence  the number of codewords in the associated code. This will enable us to give an explicit algebraic expression for the rate of the code, something which has not been possible with former constructions using BIBDs (see for example \cite{VasMil04}). Further we will show that there always exists a set of $6$ linearly dependent columns in $H$ but no set of $5$ or less linearly dependent columns when $n$ is odd. Thus the distance of the code is $6$ when $n$ is odd and $4$ when $n$ is even.

But before we do this we draw the connection with the above PBIBD and pseudo orthogonal Latin squares. This connection  will aide the reader in the proofs given in Section \ref{LI}.

\section{Pseudo-orthogonal Latin squares}

A {\em Latin square} of order $m$  is an $m\times m$ array  in which each of the symbols of ${\mathbb Z}_m$ occurs once in every row and once in every column.
 Two Latin squares $A=[a(i,j)]$ and $B=[b(i,j)]$, of order $m$, are said to be {\em orthogonal} if
$$
O=\{(a(i,j),b(i,j))\mid 0\leq i,j\leq m-1\}={\mathbb Z}_m\times {\mathbb Z}_m.
$$
A set of $t$ Latin squares of order $m$ are said to be {\em mutually orthogonal}, $t$-MOLS$(m)$, if they are pairwise orthogonal.
It is well known that difference matrices can be used to construct sets of mutually orthogonal Latin squares, see for instance  \cite[Lemma 6.12]{HSS}.

To obtain a broader class of structures the orthogonality condition of MOLSs has been varied    to that of pseudo-orthogonal, see \cite{RSS} and  \cite{BB}. A pair of Latin squares, $A=[a(i,j)]$ and $B=[b(i,j)]$, of order $m$, is said to be {\em pseudo-orthogonal} if
given
$O=\{(a(i,j),b(i,j))\mid 0\leq i,j\leq m-1\}$,
 for all $a\in {\mathbb Z}_m$
$$
|\{(a,b(i,j))\mid (a,b(i,j))\in O\}|=m-1.
$$
That is, each symbol in $A$ is paired with every symbol in $B$ precisely once, except for one symbol with which it is paired twice and one symbol with which it is not paired at all. A set of $t$ Latin squares, of order $m$, are said to be mutually pseudo-orthogonal if they are pairwise pseudo-orthogonal.

The value and applicability  of pseudo-orthogonal Latin squares  has been established through applications to multi-factor crossover designs in animal husbandry \cite{BB}.  Mutually pseudo-orthogonal Latin squares can be constructed from  cyclic DCA$(k,2n;2n)$, see \cite{DDHKR,RSS}, an idea that is formalised in the next theorem.

 \begin{theorem}\cite{DDHKR}
If there exists a cyclic DCA$(k+1, 2n+1;2n)$, $Q=[q(i,j)]$, that satisfies  P1. and P2., then there exists a set of $k$ pseudo-orthogonal Latin squares of order $2n$.
\end{theorem}

The construction takes each non-zero column of $Q$ as the first column of a $2n\times 2n$ array, with subsequent columns obtained by adding 1 modulo $2n$ to the entries in the previous column. Thus, the DCA$^*(3,2n;2n)$  $X=[x(j,g)]$ given in Equation \ref{DCA} gives the following associated pair of pseudo-orthogonal Latin squares of order $6$:

\begin{center}
\begin{tabular}{cc}
$Y=[y(i,j)]$&$Z=[z(i,j)]$\\
\begin{tabular}{|c|c|c|c|c|c|}
\hline
0&1&2&3&4&5\\
\hline
1&2&3&4&5&0\\
\hline
2&3&4&5&0&1\\
\hline
3&4&5&0&1&2\\
\hline
4&5&0&1&2&3\\
\hline
5&0&1&2&3&4\\
\hline
\end{tabular}&
\begin{tabular}{|c|c|c|c|c|c|}
\hline
1&2&3&4&5&0\\
\hline
3&4&5&0&1&2\\
\hline
5&0&1&2&3&4\\
\hline
0&1&2&3&4&5\\
\hline
2&3&4&5&0&1\\
\hline
4&5&0&1&2&3\\
\hline
\end{tabular}
\end{tabular}
\end{center}

If we rewrite these Latin squares in terms of group divisible designs (see \cite{colbourn} for definition) with sets of blocks

 \begin{eqnarray*}
O_{j} &=& \{ \{ a, y(j,a)+2n, z(j,a)+4n\}
\mid 0\leq a\leq 2n-1,a\neq n\},\ j\in {\Bbb Z}_{2n}\setminus\{n\}
   \end{eqnarray*}
 then we obtain a collection of $2n-1$ parallel classes forming a PBIBD with block size 3 and $\lambda$ equal to 0 or 1 and block set ${\cal B}=\cup_{j} O_j$ as before.

\section{LDPC Codes from PBIBDs}\label{LI}

In this section we show that the incidence matrix as defined in Equation \ref{IM} defines a $6n\times (4n^2-2n)$ parity-check matrix $H$ of an LPDC code with $\rho=2n-1$, $\gamma=3$, with no cycles of length less than 6, and rank $R=6n-2$. The dimension of the null space of $H$ is $4n^2-8n+2$ and the minimum distance $d=6$ when $n$ is odd and 4 when $n$ is even.

We start by showing that the girth of the Tanner graph is at least six.

\begin{lemma}
The LDPC code defined by the parity-check matrix $H$ given by Equation \ref{IM} has girth at least 6.
\end{lemma}

\begin{proof}
Since each block of ${\cal B}$ has precisely 3 entries, every column of $H$ sums to $\gamma=3$. There are $2n-1$  parallel classes and each element of $V$ is contained in precisely one block of each parallel class. Thus each row of $H$ will sum to $\rho=2n-1$.  For any pair $y,z\in V,$ $\lambda_{y,z}\leq 1$ hence the inner product of any two rows of the incidence matrix is less than or equal to 1. If the inner product of any two columns is greater than or equal to 1 then there exists two blocks of $B$ which intersect in two or more elements but this contradicts the fact that  $\lambda_{y,z}\leq 1$ for all $y,z\in V$. Thus the Tanner graph has no cycles of length less than 6.
\end{proof}

Next we show that  the matrix $H$ as given in Equation \ref{IM} has rank exactly  $6n - 2$. Recall that if $A$ is a matrix with $m$ columns then $\mbox{rank}(A)+\mbox{nullity}(A)=m.\label{rank}$ Hence the proof of this part requires two steps - first we show in Lemma \ref{atmost} that the rank is at most $6n - 2$ by using the rows of the $H$, and then we give a set of $6n - 2$ columns that are linearly independent.

\begin{lemma}\label{atmost}
Working over ${\Bbb Z}_2$, the rank of $H$ is at most $6n - 2$.
\end{lemma}
\begin{proof}

Since the size of $H$ is $6n\times (4n^2-2n)$ the rank of $H$ is at most $6n$. First we will find two linearly dependent rows in the parity check matrix showing that the rank of $H$ is at most $6n-2$. Remember that each block contains exactly one element from each of the sets $\{0,...,2n-1\}$, $\{2n,...,4n-1\}$ and $\{4n,...,6n-1\}$. The set of rows $R_{1,2}=\{0,...,4n-1\}$, $R_{2,3}=\{2n,...,6n-1\}$ and $R_{1,3}=\{0,...,2n-1,4n,...,6n-1\}$  are each linearly dependent since the column sum of $H$ restricted to any of these sets of rows is $2$. Further,  for each of these sets any row can be written as the sum of the other  $4n-1$ rows and we require at least two of these sets to cover all rows, thus there are at least 2 linearly dependent rows.
\end{proof}

Next we need to show that $H$ contains a linearly independent set of columns (blocks) of size  $6n-2$, giving the rank of $H$ as exactly $6n-2$. We give the proof for $n\geq 6$, though the proof follows similarly for smaller sizes, which can be seen by direct calculation.
\vspace{3mm}

Let ${\cal I} =C_1\cup C_2 \cup C_3 \cup C_4$ be the set of blocks where
\begin{eqnarray*}
C_1&=&\{\{a,2n+a,4n+(1+a \mbox{ mod } 2n)\}\mid a=0,\dots, 2n-1\},\\
C_2&=&\{\{b,2n+1+b,4n+(3+b \mbox{ mod } 2n)\}\mid b=0,\dots, 2n-2\},\\
C_3&=&\{\{c,2n+2+c,4n+(5+c \mbox{ mod } 2n)\}\mid c=0,\dots, 2n-4\},\\
C_4&=&\{\{0,2n+n+1,4n+2\},\{1,2n+n+2,4n+3\}\}.\\
\end{eqnarray*}
Let $E=\{x\in B \mid  B\in {\cal I}\}$. Note that elements of $E$ have replication number $r=1,2,3,4$ in the blocks of ${\cal I}$.
The following table categorizes the elements of $E$ according to their replication number, for
$n \geq 6$. Note that the replication numbers are only very slightly different for $n =3, 4,5$.

\begin{center}
\begin{tabular}{l|l|l|l}
\multicolumn{4}{c}{Replication Numbers for $n \geq 5$}\\
\hline
 1 & 2 & 3 & 4\\
 \hline
 $2n-1, 2n.$ & $2n-3,2n-2,$  & $2,\dots,2n-4,$    & $0,1$\\
                & $2n+1,4n-1,$ & $2n+2,\dots,3n,$ & $3n+1, 3n+2.$\\
                &  $4n+2, 4n+4$&$3n+3, \dots, 4n-2,$ &\\
                &                     &$4n,4n+1,4n+3,$ &                    \\
                &                     & $4n+5, \dots, 6n-1$.&                 \\
        \hline
\end{tabular}
\end{center}

\begin{example}
The block sets $C_1, C_2, C_3$ and $C_4$ are given for $n = 6$. The elements of the blocks are arranged in different rows to make it easy to see the repetition in the elements. Each block is read off by reading corresponding cells in groups of rows. For example, block $\{0,12,25\} \in C_1$ is read from rows 1, 5 and 9.
\[ \begin{array}{lllllllllllll}
C_1 & 0 & 1 & 2 & 3 & 4 & 5 & 6 & 7 & 8 & 9&10&11 \\
C_2 & 0 & 1 & 2 & 3 & 4 & 5 & 6 & 7 & 8 & 9 &10& \\
C_3 & 0 & 1 & 2 & 3 & 4 & 5 & 6 &7  &8  &&& \\
C_4 & 0 & 1 &&&&&&&&&&\\\hline
C_1 &  12 & 13 & 14 & 15 & 16 & 17 & 18 & 19 & 20 & 21 & 22 & 23\\
C_2 &  13 & 14 & 15 & 16 & 17 & 18 & 19 & 20 & 21 & 22 & 23 &    \\
C_3 &  14 & 15 & 16 & 17 & 18 & 19 & 20 & 21 & 22 &      &    &\\
C_4  & 19 & 20 &&&&&&&&&&\\\hline
C_1 &  25 & 26 & 27 & 28 & 29 & 30 & 31 & 32 & 33 & 34 & 35 & 24\\
C_2 &  27 & 28 & 29 & 30 & 31 & 32 & 33 & 34 & 35 & 24 & 25 &    \\
C_3 & 29 & 30 & 31 & 32 & 33 & 34 & 35 & 24 & 25 &       &    &\\
C_4  & 26 & 27 &&&&&&&&&&\\\hline
\end{array}
\]
\end{example}

We claim that the set of blocks ${\cal I}$ correspond to a linearly independent set of columns of $H$, giving the rank of $H$ as exactly $6n -2$.  We prove this result in Proposition \ref{LinDep} where we show that there is no subset of $\cal{I}$ corresponds to a  linearly dependent set of columns.

\begin{proposition}\label{LinDep}
Define $H$ and ${\cal I}$ as above. The columns of $H$ corresponding to  the blocks of ${\cal I}$ are a linearly independent set of vectors over ${\Bbb Z}_2^{6n}$, implying the rank of $H$ is exactly $6n - 2$.
\end{proposition}

  \begin{proof}
Let ${\cal D}\subset {\cal I}$, with  $D$ representing to corresponding set of columns in $H$. We will assume that $D$ is a linearly dependent set and show that this leads to a contradiction. The assumption that $D$ is a linearly dependent set implies that any row $y$ restricted to the  columns of $D$  contain 0 or  an even number of 1s.

Two blocks $\{0,2n,4n+1\}$ and $\{2n-1, 4n-1,4n\}$ in $C_1$ ($a = 0$ and $a = n-1$), respectively, contain the entries  $2n-1$ and $2n$ with replication number 1 in ${\cal I}$. This precludes these two blocks from being in ${\cal D}$. This immediately results in $\{2n-2, 4n-1, 4n+1\}$ in $ C_2$  ($b = 2n-2$) being excluded from ${\cal D}$  since point $ 4n - 1$ would have replication number 1 in ${\cal D}$ and likewise $\{2n-2, 2n+ (2n-2), 4n + (1 + 2n-2\, \mbox{ mod } 2n) \}$ of $C_1$ ($a = 2n-2$)  being excluded from ${\cal D}$.

Proceeding in this manner, we sequentially take the triple of values $c,\ b,\ a$, for $c = 2n - 4, 2n-5, \dots, n; \hspace{3mm}  b = 2n-3, 2n-4, \dots, n+2; \hspace{3mm}   a = 2n-3, 2n-4, \dots, n+2$, and exclude from ${\cal D}$ the corresponding blocks   from $C_3$,  $C_2$, and $C_1$ respectively.

Thus if $D$ exists, then ${\cal D}$ must be a subset of ${\cal I}^\prime = C_1^\prime \cup  C_2^\prime \cup C_3^\prime \cup C_4$, where
\begin{eqnarray*}
C_1^\prime&=&\{\{a,2n+a,4n+(1+a \mbox{ mod } 2n)\}\mid a=1,\dots, n+1\},\\
C_2^\prime&=&\{\{b,2n+1+b,4n+(3+b \mbox{ mod } 2n)\}\mid b=0,\dots, n+1\},\\
C_3^\prime&=&\{\{c,2n+2+c,4n+(5+c \mbox{ mod } 2n)\}\mid c=0,\dots, n-1\},\\
C_4&=&\{\{0,2n+n+1,4n+2\},\{1,2n+n+2,4n+3\}\},\\
\end{eqnarray*}
where $C_i^\prime \subset C_i$ for $i = 1, \dots, 3$. Let $E^\prime =\{x\in B \mid  B\in {\cal I}^\prime\}$. The replication numbers of $x \in E^\prime$ for $n \geq 6$ are as in the table below. Note that there are no singly occuring elements in $E^\prime$.

\begin{center}
\begin{tabular}{l|l|l}
\multicolumn{3}{c}
{Replication Number}\\
\hline
 2 & 3 & 4\\
 \hline
$n, n+1,$&$0, 2, 3, \dots, n - 1,$&$1,$\\
 $2n + 1, 3n+2,$      & $2n+2,\dots, 3n,$&$3n+1$\\
 $4n+2, 4n+4, 5n+3, 5n+4$&$4n+3,4n+5,\dots, 5n+2$  &\\
       \hline
\end{tabular}
\end{center}

We next argue the following four cases: Case 1, both $\{0,3n+1,4n+2\},\{1,3n+2,4n+3\}\in {\cal D}$; Case 2,  $\{0,3n+1,4n+2\}\in {\cal D}$ and $\{1,3n+2,4n+3\}\notin {\cal D}$;  Case 3, $\{1,3n+2,4n+3\}\in {\cal D}$ and $\{0,3n+1,4n+2\}\notin {\cal D}$;  Case 4, $\{0,3n+1,4n+2\},\{1,3n+2,4n+3\}\notin {\cal D}$, and show that every case leads to a contradiction.

{\bf Case 1}: Suppose that both $\{0,3n+1,4n+2\},\{1,3n+2,4n+3\}\in {\cal D}$. This assumption implies $\{1,2n+1,4n+2\},\{(0,2n+1,4n+3\}\in {\cal D}$. This results in
$\{e,2n+2+e,4n+5+e\},\{e+2,2n+2+e,4n+3+e\},\{e+1,2n+2+e,4n+4+e\}\notin {\cal D}$ consecutively, for $e=0,\dots, n-2$.

This leaves only the entries $\{n-1,3n+1,5n+4\},\{n,3n+1,5n+3\},\{n+1,3n+1,5n+2\},\{n+1,3n+2,5n+4\}$.
If $D$ is to be a dependent set, the block $\{n+1,3n+2,5n+4\}$ as well as one of the other three blocks must be in ${\cal D}$.
But no matter which of the three blocks is chosen it results in one of $n, n-1, n+1, 5n+ 2, 5n+3, 5n+4 \in E^\prime$ occurring in a single block in ${\cal D}$, a contradiction.

 {\bf Case 2}: Suppose $\{0,3n+1,4n+2\}\in {\cal D}$ but $\{1,3n+2,4n+3\}\notin {\cal D}$. This implies  $\{1,2n+1,4n+2\},\{0,2n+1,4n+3\}\in {\cal D}$ and subsequently that $\{0,2n+2,4n+5\}\notin {\cal D}$. Then for $e=0,\dots, n-2$, $\{e+2,2n+2+e,4n+3+e\},\{e+1,2n+2+e,4n+4+e\}\in {\cal D}$ and $\{e+1,2n+3+e,4n+6+e\}\notin {\cal D}$.

This only leaves the blocks $\{n,3n+1,5n+3\}, \{n+1,3n+1,5n+2\},\{n+1,3n+2,5n+4\}$. Now $ \{n+1,3n+2,5n+4\}$ cannot be in ${\cal D}$ as $3n+2 \in E^\prime$ would then occur in only one block in ${\cal D}$. To get a matching pair to the block containing $3n + 1 \mbox{ in } {\cal D}$, we need only one of the other two blocks. But whichever choice we make results in one of $n, n+1, 5n+3 \in E^\prime$ occurring in only
 one block of ${\cal D}$, a contradiction to $D$ being linearly dependent.

Finally reviewing the blocks that are in ${\cal D}$ we see that $5n+2\in E^\prime$ now occurs in only one block of ${\cal D}$ which is also a contradiction. Hence $D$ cannot be linearly dependent.

{\bf Case 3}: Suppose $\{1,3n+2,4n+3\}\in {\cal D}$  but $\{0,3n+1,4n+2\}\notin {\cal D}$. Then $\{1,2n+1,4n+2\}\notin {\cal D}$ and consequently $\{0,2n+1,4n+3\}\notin {\cal D}$ and $\{0,2n+2,4n+5\}\notin {\cal D}$. Then, arguing similarly to case 2,   $\{e+2,2n+2+e,4n+3+e\},\{e+1,2n+2+e,4n+4+e\}\in D$ and $\{e+1,2n+3+e,4n+6+e\}\notin D$ for $e=0,\dots, n-2$, and the result follows.

   {\bf  Case 4}: Suppose that both $\{0,3n+1,4n+2\},\{1,3n+2,4n+3\}\notin {\cal D}$. Then it follows that $\{1,2n+1,4n+2\}\notin {\cal D}$ and so $\{0,2n+1,4n+3\}\notin {\cal D}$ and  for $e=0,\dots, n-2$, this implies
$\{e,2n+2+e,4n+5+e\},\{e+2,2n+2+e,4n+3+e\},\{e+1,2n+2+e,4n+4+e\}\notin {\cal D}$, reducing as in the previous cases to a contradiction.

Thus the columns of $H$ that correspond to ${\cal I}$ form a  linearly independent set of vectors and the rank of $H$ is $6n - 2$.
\end{proof}

Using this information we see that the rate of the code  is $ (b-\mbox{rank}(H))/b = (4n^2-8n+2)/(4n^2-2n)$
and that $\mbox{nullity}(H) =  4n^2-8n+2$, which is the number of codewords in the LDPC code. The table below lists the rate of the code for $6 \leq n \leq 15$. The codes constructed here achieve high rate for much shorter lengths than previous LDPC codes obtained by combinatorial constructions. 

\begin{center}
\begin{tabular}{|l|p{2.25cm}|p{2cm}|p{1.5cm}||l|p{2.25cm}|p{2cm}|p{1.5cm}|}
$n$&Code length & Code dim   & Rate of  & $n$&Code length & Code dim   & Rate of\\
   & $ 4n^2 -2n$       &  $4n^2 - 8n +2$ & code & & $ 4n^2 -2n$       &  $4n^2 - 8n +2$ & code\\\hline
6 & 132 & 98 & 0.742&11& 462 & 398 & 0.861\\
7 & 182 & 142 & 0.780 & 12 & 552 & 482 & 0.873 \\ 
8 & 240 & 194 & 0.808& 13 & 650 & 574 & 0.883 \\
9& 306& 254 & 0.830& 14 & 756 & 674 & 0.891  \\
10 & 380 & 322 & 0.847 & 15 & 870 & 782 & 0.899\\
\end{tabular}
\end{center}

Finally we show that, when $n$ is odd the minimum distance of the code is 6, by showing that there exists 6 columns of $H$ that are linearly dependent, but no 5 columns are linearly dependent. From  \cite[Theorem 3.1]{VasMil04} the LDPC code with parity-check matrix $H$ as given in Equation \ref{IM} has minimum distance $d \leq 6$. The next lemma establishes the exact value of $d$ for the given parity-check matrix $H$.

\begin{lemma}
The LDPC code with parity-check matrix $H$ as given by Equation \ref{IM} has minimum distance 6 when $n$ is odd, and 4 when $n$ is even.
\end{lemma}

\begin{proof}

It is also easy to see that the set of columns of $H$ corresponding to the blocks
\begin{eqnarray*}
\begin{array}{ll}
\{0,2n+1,4n+3\} & \{0,2n+2,4n+5\}\\
\{1,2n+2,4n+4\} & \{1, 3n+2,4n+3\}\\
\{2n-1,2n+1,4n+4\} &\{2n-1,3n+2,4n+5\}
\end{array}
\end{eqnarray*}
is a linearly dependent set of columns. Thus there exists a codeword of weight $6$.

Now assume that there exists a set $D$ of $5$ or less linearly dependent columns with corresponding set of blocks denoted ${\cal D}$. The sum of any row of $H$ restricted to the columns of $D$  must be even, thus $D$ is made up of precisely four linearly dependent columns and consequently ${\cal D}=\{\{0,y+2n,x(y,2)+4n\},\{0,z+2n,x(z,2)+4n\},\{c,y+2n,x(z,2)+4n\},\{c,z+2n,x(y,2)+4n\}\}$, $1\leq c\leq 2n-1$ and $y,z\in \{0,1,\dots, 2n-1\}$.

When $n$ is even the set of columns of $H$ corresponding to the following blocks is a linearly dependent set:
\begin{eqnarray*}
\begin{array}{llll}
\{0,2n,4n+1\}, & \{0,2n+n/2,5n+1\}, & \{n-1, 2n, 5n+1\}, & \{n-1, 2n+ n/2, 4n+1\}.
\end{array}
\end{eqnarray*}

Thus assume that $n$ is odd.

The orbit structure  of the PBIBD implies the elements of a block of ${\cal D}$ satisfy  one of two equations; that is, if  $\{c,y+2n,x(z,2)+4n\}\in{\cal D}$, then
\begin{eqnarray}\label{constraint}
x(z,2)\equiv 2y+1-c\,(\mbox{mod }2n)\mbox{ or }x(z,2)\equiv 2(y-n)-c\,(\mbox{mod }2n).
\end{eqnarray}

The argument is now split into four cases 1), 2), 3) and 4) as set out in Table 1, and  in each case applying Equations \eqref{constraint} leads to four subcases a), b), c) and d) as set out in Table 2.
\begin{center}
\begin{tabular}{cc}
Table 1& Table 2\\
\begin{tabular}{|r|l|l|}
\hline
Case &$x(y,2)$&$x(z,2)$\\
\hline
1)&$2y+1$&$2z+1$\\
\hline
2)&$2y+1$&$2(z-n)$\\
\hline
3)&$2(y-n)$&$2z+1$\\
\hline
4)&$2(y-n)$&$2(z-n)$\\
\hline
\end{tabular}&
\begin{tabular}{|r|l|l|}
\hline
Subcase &$x(z,2)$&$x(y,2)$\\
\hline
a)&$2y+1-c\,(\mbox{mod }2n)$&$2z+1-c\,(\mbox{mod }2n)$\\
\hline
b)&$2y+1-c\,(\mbox{mod }2n)$&$2(z-n)-c\,(\mbox{mod }2n)$\\
\hline
c)&$2(y-n)-c\,(\mbox{mod }2n)$&$2z+1-c\,(\mbox{mod }2n)$\\
\hline
d)&$2(y-n)-c\,(\mbox{mod }2n)$&$2(z-n)-c\,(\mbox{mod }2n)$\\
\hline
\end{tabular}
\end{tabular}
\end{center}
Without loss of generality we may assume $0\leq y<z<2n-1$. Hence Case $3)$ becomes redundant.

For Case 1), we have $x(y,2)=2y+1$ and $x(z,2)=2z+1$, implying $0\leq y<z\leq n-1$. Hence $0<z-y\leq n-1$ and
\begin{eqnarray}\label{4n}
0<4(z-y)<4n.
\end{eqnarray}
In  Subcase 1a) we equate terms and subtract Equations \eqref{first1a} and \eqref{second1a}  to obtain Equation \eqref{third1a}:
\begin{eqnarray}
x(z,2)=2z+1&\equiv& 2y+1-c\,(\mbox{mod }2n)\label{first1a}\\
x(y,2)=2y+1&\equiv& 2z+1-c\,(\mbox{mod }2n)\label{second1a},\\
x(z,2)-x(y,2)=2(z-y)&\equiv& 2(y-z)\,(\mbox{mod }2n)\label{third1a},\\
4(z-y)&\equiv& 0\,(\mbox{mod }2n)
\end{eqnarray}
Hence $4(z-y)=2n$\ by Equation (\ref{4n}).  Leading to a contradiction, as $n$ is odd. This same argument applies in Subcase 1d).

Using similar arguments in
Subcase 1b):
\begin{eqnarray*}
x(z,2)=2z+1&\equiv& 2y+1-c\,(\mbox{mod }2n)\\
x(y,2)=2y+1&\equiv& 2(z-n)-c\,(\mbox{mod }2n),\\
x(z,2)-x(y,2)=2(z-y)&\equiv& 2(y-z)+1\,(\mbox{mod }2n),\\
4(z-y)&\equiv& 1\,(\mbox{mod }2n)
\end{eqnarray*}
implying $4(z-y)=2n\alpha+1$ for some $\alpha\in Z$ which is a contradiction.
This same argument applies in Subcases 1c), 2a), 2d), 4b) and 4c).

Now for Subcase 2b) summing up Equations \eqref{first2b} and \eqref{second2b} we have
\begin{eqnarray}
x(z,2)=2(z-n)&\equiv& 2y+1-c\,(\mbox{mod }2n),\label{first2b}\\
x(y,2)=2y+1&\equiv& 2(z-n)-c\,(\mbox{mod }2n),\label{second2b}\\
2(z+y)+1&\equiv& 2(z+y)+1-2c\,(\mbox{mod }2n),\label{third2b}
\end{eqnarray}
implying $c=n$. Hence $\{n,y+2n,x(z,2)+4n\}$=$\{n,y+2n,2(z-n)+4n\}$ is a block of ${\cal D}$. We have assumed $x(y,2)=2y+1$ thus $0\leq y\leq n-1$. 
Now $y-n \equiv\, y+n\,(\mbox{mod }2n)$  so there exists a column in $H$ corresponding to the block $\{0,y+n+2n,2(z-n)-n\,(\mbox{mod }2n)+4n\}$. Then $x(y+n,2)=2(y+n-n)\equiv\,2(z-n)-n\,(\mbox{mod }2n)$. This implies $2(y-z)\equiv\, n\, (\mbox{mod }2n)$. Leading to a contradiction, as $n$ is odd.

A similarly argument works for Subcase 2c).

For Case 4), we have $x(y,2)=2(y-n)$ and $x(z,2)=2(z-n)$ so $n+1\leq y<z\leq 2n-1$. Hence $0<z-y\leq n-1$ and we have $0<4(z-y)<4n$ as in Equation \eqref{4n}. So Subcases 4a) and 4d) follow as in Subcase 1a)

\end{proof}

Thus we have proved the following theorem:

\begin{theorem}\label{LDPC_pseudo}
Let $H$ be the incidence matrix given by Equation \ref{IM}. Then $H$ is the parity-check matrix of an LDPC code of length $4n^2 - 2n$, girth at least 6,   rate $(4n^2 - 8n  + 2)/(4n^2 - 2n)$ and  minimum distance $d = 6$ when $n$ is odd. When $n$ is even, the code has minimum distance $d =4$.
\end{theorem}

\section{Conclusion}
In this paper we constructed an infinite family of LDPC codes from PBIBDs for $n \in {\Bbb Z}$. We showed that these codes have Tanner graphs of girth at least 6 and have minimum distance 6 when $n$ is odd. Unlike previous combinatorial constructions we were able to give, explicitly, the rate of the code. In addition, in the past,  combinatorial codes have needed to be quite long to get high rates ($\geq 0.8$).  The codes we obtain have rate $(4n^2 - 8n + 2)/ (4n^2 - 2n)$, and achieve rate $\geq 0.8$ for $n \geq 8$.

\end{document}